\documentclass[11pt,reqno,british,a4paper,twoside]{amsart}

\usepackage[utf8]{inputenc}
\usepackage{lmodern}
\usepackage[T1]{fontenc}
\DeclareMathAlphabet{\mathsl}{OT1}{lmr}{m}{sl}
\DeclareMathSymbol{\leqslant}{\mathrel}{symbols}{"AC}
\DeclareMathSymbol{\geqslant}{\mathrel}{symbols}{"AD}

\usepackage{babel,mathtools,paralist,xspace,stmaryrd,url}
\usepackage{booktabs,array,fixltx2e,fancyhdr,csquotes}
\usepackage[babel=true]{microtype}
\usepackage[margin=30pt,font=small]{caption}

\rmfamily

\mathtoolsset{mathic=true}
\setdefaultenum{\upshape (i)}{}{}{}

\theoremstyle{plain}
\newtheorem{theorem}{Theorem}[section]
\newtheorem{proposition}[theorem]{Proposition}
\newtheorem{lemma}[theorem]{Lemma}
\newtheorem{corollary}[theorem]{Corollary}

\theoremstyle{definition}

\theoremstyle{remark}
\newtheorem{remark}[theorem]{Remark}

\numberwithin{equation}{section}

\newcommand{\Lie}[1]{{\mathsl{#1}}}
\newcommand{\lie}[1]{{\mathfrak{#1}}}

\newcommand{\SO}{\Lie{SO}}
\newcommand{\so}{\lie{so}}

\newcommand{\Sp}{\Lie{Sp}}
\newcommand{\sP}{\lie{sp}}
\newcommand{\SU}{\Lie{SU}}
\newcommand{\su}{\lie{su}}
\newcommand{\Un}{\Lie{U}}
\newcommand{\un}{\lie{u}}

\newcommand{\bC}{{\mathbb C}}

\newcommand{\bR}{{\mathbb R}}

\DeclareMathOperator{\CP}{\bC P}
\DeclareMathOperator{\End}{End}
\DeclareMathOperator{\Gro}{\widetilde{Gr}}
\DeclareMathOperator{\Id}{Id}

\newcommand{\cyc}{\mathop{\hbox{\larger\(\mathfrak{S}\)}}}

\newcommand{\NB}{\nabla}
\newcommand{\LC}{\nabla^{\mathrm{LC}}}
\newcommand{\cQ}{\mathcal{Q}}

\DeclarePairedDelimiterX{\inp}[2]{\langle}{\rangle}{#1, #2}

\DeclarePairedDelimiter{\base}{\{}{\}}
\DeclarePairedDelimiterX{\Set}[2]{\{}{\}}{\, #1 \,\delimsize\vert\, #2 \,}
\DeclarePairedDelimiter{\rcomp}{\llbracket}{\rrbracket}
\DeclarePairedDelimiter{\real}{\lbrack}{\rbrack}

\newcommand{\eqbreak}[1][2]{\\&\hskip#1em}
\newcommand{\hyphen}{\nobreak-\nobreak\hspace{0pt}}

\begin{document}

\title[Quaternion Geometries on Twistor Space]{Quaternion Geometries
on the Twistor Space of the Six-Sphere} 

\author{Francisco Martín Cabrera}
\address[F.~Martín Cabrera]{Department of Fundamental Mathematics\\
University of La Laguna\\ 38200 La Laguna\\ Tenerife\\ Spain}
\email{fmartin@ull.es}

\author{Andrew Swann}
\address[A.\,F.~Swann]{Department of Mathematics\\
Aarhus University\\
Ny Munkegade 118\\ Bldg
  1530\\ DK-8000 Aarhus C\\ Denmark \textit{and}
  CP\textsuperscript3-Origins Centre of Excellence for Cosmology and
  Particle Physics Phenomenology\\ University of Southern Denmark\\
  Campusvej 55\\ DK-5230 Odense M\\ Denmark}
\email{swann@imf.au.dk}


\begin{abstract}
  We explicitly describe all \( \SO(7) \)-invariant almost
  quaternion\hyphen Hermitian structures on the twistor space of the
  six sphere and determine the types of their intrinsic torsion.
\end{abstract}


\maketitle

\begin{center}
  \begin{minipage}{0.7\linewidth}
    \small
    \tableofcontents
  \end{minipage}
\end{center}

\section{Introduction}
\label{sec:introduction}
Recently Moroianu, Pilca and Semmelmann~\cite{Moroianu-PS:hom-aqH}
found that the twistor space \( M = \SO(7)/{\Un(3)} \) of the six
sphere~\( S^6 \) admits a homogeneous almost quaternion\hyphen
Hermitian structure.  This arose as part of their striking result that
\( M \) is the only such homogeneous space with non-zero Euler
characteristic that is neither quaternionic Kähler (the quaternionic
symmetric spaces of Wolf~\cite{Wolf:quaternionic}) nor~\( S^2\times
S^2 \).

In this paper we show that there is exactly a one-dimensional family
of invariant almost quaternion\hyphen Hermitian structures on \( M \),
with fixed volume, and determine the types of their intrinsic torsion.
We will see that the family contains inequivalent structures, and
includes the symmetric Kähler metric of the quadric \( \Gro_2(\bR^6) =
\SO(8)/{\SO(2)\SO(6)} \).  Each member of the family will be shown to
have almost quaternion\hyphen Hermitian type \( \Lambda^3_0E(S^3H + H)
\) with the first component non-zero, confirming that they are not
quaternionic Kähler; one member of the family has pure type \(
\Lambda^3_0ES^3H \), and this is the first known example of such a
geometry.  However, the structure singled out by this almost
quaternionic\hyphen Hermitian intrinsic torsion is not the Kähler
metric of the quadric nor the squashed Einstein metric in the
canonical variation.

\subsubsection*{Acknowledgements.} This work is partially supported by
the Danish Council for Independent Research, Natural Sciences and
MICINN (Spain) Grant MTM2009-13383.  We thank the authors of
\cite{Moroianu-PS:hom-aqH} for informing us of their work and in
particular Andrei Moroianu for useful conversations.

\section{Invariant forms}

The subgroup \( \Un(3) \) of \( \SO(7) \) arises from a choice of
identification of \( \bR^7 \) as \( \bR \oplus \bC^3 \).  Regarding \(
\Un(3) \) as \( \Un(1)\SU(3) \), we may write \( \bC^3 = \bR^6 =
\rcomp{L\lambda^{1,0}} \), meaning that \( \bR^6 \otimes \bC =
L\lambda^{1,0} + \overline{L\lambda^{1,0}} \cong L\lambda^{1,0} +
L^{-1}\lambda^{0,1} \), where \( L = \bC \) and \( \lambda^{1,0} =
\bC^3 \) as the standard representations of \( \Un(1) \) and \( \SU(3)
\), respectively.  We thus have \( \Un(3) \leqslant \SO(6) \leqslant
\SO(7) \), so \( M = \SO(7)/{\Un(3)} \) fibres over \( S^6 =
\SO(7)/{\SO(6)} \) with fibre \( \SO(6)/{\Un(3)} \), the almost
complex structures on~\( T_xS^6 \).  Thus \( M \) is the (Riemannian)
twistor space of~\( S^6 \).

Since \( \lambda^{3,0} = \Lambda^3\lambda^{1,0} = \bC \) is trivial,
we have \( \lambda^{2,0}\cong \lambda^{0,1} \) as \( \SU(3)
\)-modules.  The Lie algebra of \( \SO(7) \) now decomposes as
\begin{equation*}
  \begin{split}
    \so(7) &= \Lambda^2\bR^7
    = \Lambda^2(\bR + \rcomp{L\lambda^{1,0}})
    = \rcomp{L\lambda^{1,0}} + \rcomp{L^2\lambda^{2,0}} +
    \real{\lambda^{1,1}} \\ 
    &\cong \rcomp{L\lambda^{1,0}} + \rcomp{L^2\lambda^{0,1}} + \un(1)
    + \su(3). 
  \end{split}
\end{equation*}
Here \( [\lambda^{1,1}] \) is the real module whose complexification
is \( \lambda^{1,1} = \lambda^{1,0} \otimes \lambda^{0,1} \); it
splits in to two irreducible modules \( [\lambda^{1,1}_0] \cong \su(3)
\) and \( \bR = \un(1) \).

We thus have that the complexified tangent space of \( M =
\SO(7)/{\Un(3)} \) is the bundle associated to
\begin{equation}
  \label{eq:TC}
  \begin{split}
    T \otimes \bC
    &= \bigl(\rcomp{L\lambda^{1,0}} + \rcomp{L^2\lambda^{0,1}}\bigr)
    \otimes \bC\\
    &= L\lambda^{1,0} + L^{-1}\lambda^{0,1} + L^2\lambda^{0,1} +
    L^{-2}\lambda^{1,0}\\
    &= (L^{1/2}\lambda^{0,1} +
    L^{-1/2}\lambda^{1,0})(L^{3/2}+L^{-3/2}). 
  \end{split}
\end{equation}
This allows us to write \( T \otimes \bC = EH \), where \( E =
L^{1/2}\lambda^{0,1} + L^{-1/2}\lambda^{1,0} \) and \( H =
L^{3/2}+L^{-3/2} \) are representations of \( \Un(1)_2\times \SU(3) \)
as a subgroup of \( \Un(1)_L\SU(3) \times \Un(1)_R \leqslant \Sp(3)
\times \Sp(1) \).  Here \( \Un(1)_2 \) is a double cover of \( \Un(1)
\) and is included in \( \Un(1)_L \times \Un(1)_R \) via the map \(
e^{i\theta} \mapsto (e^{-i\theta},e^{3i\theta}) \).  In this way, we
see that \( M = \SO(7)/{\Un(3)} \) carries an invariant \(
\Sp(3)\Sp(1) \)-structure, where \( \Sp(3)\Sp(1) = (\Sp(3) \times
\Sp(1)) / \{\pm(1,1)\} \).  This is the \( G \)-structure description
of an almost quaternion\hyphen Hermitian structure.

Geometrically an almost quaternion\hyphen Hermitian structure is
specified by a Riemannian metric \( g \) and a three-dimensional
subbundle \( \mathcal G \) of \( \End(TM) \) which locally has a basis
\( I \), \( J \), \( K \) satisfying the quaternion identities
\begin{equation*}
  I^2 = -1 = J^2,\qquad IJ = K = - JI
\end{equation*}
and the compatibility conditions
\begin{equation*}
  g(I\cdot,I\cdot) = g(\cdot,\cdot) = g(J\cdot,J\cdot).
\end{equation*}
There are then local two-forms
\begin{gather*}
  \omega_I(X,Y) = g(X,IY),\quad
  \omega_J(X,Y) = g(X,JY),\\
  \omega_K(X,Y) = g(X,KY)
\end{gather*}
and with the local form \( \omega_c = \omega_J +i\omega_K \) of type
\( (2,0) \) with respect to~\( I \).  Since they are non-degenerate,
the local forms \( \omega_I \), \( \omega_J \), \( \omega_K \) are
sufficient to determine the local almost complex structures \( I \),
\( J \) and \( K \) and the metric~\( g \).

Equation~\eqref{eq:TC}, show us that \( T \) has two inequivalent
irreducible summands \( \rcomp{L\lambda^{1,0}} \) and \(
\rcomp{L^2\lambda^{0,1}} \) and so there are two invariant forms \(
\omega_0 \) and \( \tilde \omega_0 \) spanning \( \Omega^2(M)^{\SO(7)}
\).  However, we have that
\begin{equation}
  \label{eq:L2T}
  \begin{split}
    \Lambda^2T &= \Lambda^2\rcomp{L\lambda^{1,0}} +
    \Lambda^2\rcomp{L^2\lambda^{0,1}} +
    \rcomp{L\lambda^{1,0}}\wedge\rcomp{L^2\lambda^{0,1}}\\
    &= (\bR\omega_0 + \real{\lambda^{1,1}_0} + \rcomp{L^2\lambda^{0,1}})
    + (\bR\tilde\omega_0 + \real{\lambda^{1,1}_0} +
    \rcomp{L^4\lambda^{1,0}}) \eqbreak
    + (\rcomp{L^3} + \rcomp{L^3}\real{\lambda^{1,1}_0} +
    \rcomp{L\lambda^{1,0}} + \rcomp{L\sigma^{0,2}}),
  \end{split}
\end{equation}
where \( \sigma^{0,2} = S^2\lambda^{0,1} \).  There is thus an
addition \( 2 \)-dimensional subspace~\( \rcomp{L^3} \) preserved by
the \( \SU(3) \)-action.  This space is spanned by local \( \SU(3)
\)-invariant forms \( \omega_J \) and \( \omega_K \), that are mixed
under the \( \Un(1) \)-action, so that \( \omega_c = \omega_J +
i\omega_K \) is a basis element of~\( L^3 \).  We may now consider the
triple of forms
\begin{equation}
  \label{eq:forms}
  \omega_I = \lambda\omega_0 + \mu\tilde\omega_0,\quad
  \omega_J\quad\text{and}\quad \omega_K
\end{equation}
which will be seen to result in an almost quaternion\hyphen Hermitian
structure when
\begin{equation}
  \label{eq:normalisation}
  20\lambda^3\mu^3(\omega_0)^3(\tilde\omega_0)^3 = (\omega_J)^6.
\end{equation}
This equation is necessary, as each two form in the triple must define
the same volume element.

We note that for an almost quaternion\hyphen Hermitian structure the
four-form \( \Omega = \omega_I^2 + \omega_J^2 + \omega_K^2 \) is
globally defined.  For an invariant structure, this form must lie in
\( \Omega^4(M)^{\SO(7)} \) which in our particular case is
four-dimensional.  Indeed the complete decomposition of \( \Lambda^4T \) in
to irreducible \( \Un(3) \)-modules is
\begin{equation*}
  \begin{split}
    \Lambda^4T &= \rcomp{L^6} +2\rcomp{L^3} + 4\bR \eqbreak +
    \rcomp{L^7\lambda^{1,0}} + 3\rcomp{L^4\lambda^{1,0}} +
    5\rcomp{L\lambda^{1,0}} + 4\rcomp{L^2\lambda^{0,1}} +
    2\rcomp{L^5\lambda^{0,1}} \eqbreak + 2\rcomp{L^2\sigma^{2,0}} +
    2\rcomp{L\sigma^{0,2}} + \rcomp{L^4\sigma^{0,2}} \eqbreak +
    \rcomp{L^3\sigma^{3,0}} + \rcomp{\sigma^{3,0}} +
    \rcomp{L^3\sigma^{0,3}} \eqbreak + \rcomp{L^6\lambda^{1,1}_0} +
    4\rcomp{L^3\lambda^{1,1}_0} + 6\real{\lambda^{1,1}_0} \eqbreak +
    \rcomp{L^4\sigma^{2,1}_0} + 2\rcomp{L^2\sigma^{2,1}_0} +
    \rcomp{L^2\sigma^{1,2}_0} + \rcomp{\sigma^{2,2}_0}.
  \end{split}
\end{equation*}
Now the four-forms \( \omega_0^2 \), \( \tilde\omega_0^2 \), \(
\omega_0\wedge\tilde\omega_0 \) and \( \omega_J^2+\omega_K^2 \) are
invariant and linearly independent, so they provide a basis for~\(
\Omega^4(M)^{\SO(7)} \).  It follows, Lemma~\ref{lem:all-aqH} below, that
any invariant almost hyperHermitian structure on~\( M \) is described
via the forms of~\eqref{eq:forms}.

\section{Intrinsic torsion}
\label{sec:intrinsic-torsion}

Given an invariant almost Hermitian structure on~\( M \), there is a
unique \( \Sp(3)\Sp(1) \)-connection \( \NB \) characterised by the
condition that the pointwise norm of its torsion is the least
possible.  More precisely, \( \NB \) is related to the Levi-Civita
connection by
\begin{equation*}
  \NB = \LC + \xi,
\end{equation*}
where \( \xi \)~is the intrinsic torsion
given~\cite{Cabrera-S:aqH-torsion} by
\begin{equation*}
  \xi_X Y = - \tfrac14 \sum_{A=I,J,K} A (\LC_X A) Y + \tfrac12
  \sum_{A=I,J,K} \lambda_A(X) A Y,  
\end{equation*}
with
\begin{equation*}
  6\lambda_I(X) = g(\LC_X\omega_J,\omega_K),
\end{equation*}
etc.  The tensor~\( \xi \) takes values in
\begin{equation*}
  \cQ = T^* \otimes (\sP(3) + \sP(1))^\bot \subset T^* \otimes \Lambda^2
  T^* 
\end{equation*}
where \( \sP(3) = \real{S^2E} \) and \( \sP(1) = \real{S^2H} \) are
the Lie algebras of \( \Sp(3) \) and \( \Sp(1) \).  Under the action
of \( \Sp(3)\Sp(1) \), the space \( \cQ \otimes \bC \) decomposes as
\begin{equation*}
  \cQ \otimes \bC = (\Lambda^3_0E + K + E)(S^3H + H)
\end{equation*}
with \( \Lambda^3_0E \) and \( K \) irreducible \( \Sp(3) \)-modules
satisfying \( \Lambda^3E = \Lambda^3_0E + E \) and \( E\otimes S^2E =
S^3E + K + E \).  The space \( \cQ \) thus has six irreducible
summands under \( \Sp(3)\Sp(1) \).

For an invariant structure on \( M = \SO(7)/\Un(3) \), the intrinsic
torsion lies in a \( \Un(3) \)-invariant submodule of~\( \cQ \).  As
\( \sP(3) = \real{S^2(L^{1/2}\lambda^{0,1})} =
\rcomp{L\sigma^{0,2}} + [\lambda^{1,1}_0] + \bR \) and \( \sP(1) =
\real{S^2(L^{3/2})} = \rcomp{L^3} + \bR \), equation~\eqref{eq:L2T},
implies that 
\begin{equation*}
  (\sP(3)+\sP(1))^\bot \cong \real{\lambda^{1,1}_0} +
  \rcomp{L^2\lambda^{0,1}} + \rcomp{L^4\lambda^{1,0}} +
  \rcomp{L^3}\real{\lambda^{1,1}_0} + \rcomp{L\lambda^{1,0}}.
\end{equation*}
Comparing with equation~\eqref{eq:TC}, we see that \(
(\sP(3)+\sP(1))^\bot \) contains a unique copy of each of the
irreducible summands of~\( T \), so \( \cQ^{\Un(3)} \) is two
dimensional.  As \( \Lambda^3(A+B) \cong \Lambda^3A + \Lambda^2A\otimes B
+ A\otimes \Lambda^2B + \Lambda^3B \), we find that
\begin{equation*}
  \Lambda^3_0E = (L^{3/2} + L^{-3/2}) + (L^{1/2}\sigma^{2,0}+L^{-1/2}\sigma^{0,2}).
\end{equation*}
The first summand is a copy of \( H \) and is also a submodule of \(
S^3H = L^{9/2} + L^{3/2} + L^{-3/2} + L^{-9/2} \).  This shows that \(
\real{\Lambda^3_0ES^3H}^{\Un(3)} \) and \(
\real{\Lambda^3_0EH}^{\Un(3)} \) are each one-dimensional, and so we
have
\begin{equation}
  \label{eq:invxi}
  \xi \in \cQ^{\Un(3)} \subset \real{\Lambda^3_0ES^3H} +
  \real{\Lambda^3_0EH}. 
\end{equation}

\section{Explicit structures}
\label{sec:explicit}

We now wish to determine the components of \( \xi \) in each of the
summands of~\eqref{eq:invxi}.  An invariant almost Hermitian structure
on~\( M \), may be described by two-forms as in~\eqref{eq:forms}.  As
\( \omega_J \) and \( \omega_K \) are only invariant under \( \SU(3)
\), they do not define global forms on~\( M \).  However, we do get
two such invariant forms on the total space of the circle bundle \( N
= \SO(7)/\SU(3) \to M = \SO(7)/\Un(3) \).

Let \( 0, 1, 2, 3, 1', 2', 3' \) be an orthonormal basis for \( \bR^7
= \bR + \bC^3 \), with \( 0 \in \bR \) and \( i1=1' \), etc.  Writing
\( 12 \) for \( 1 \wedge 2 \), a standard basis for \(
\rcomp{L\lambda^{1,0}} \subset \so(7) \) is given by
\begin{equation*}
  A = 01,\quad B = 02,\quad C = 03,\quad A' = 01',\quad B' = 02',\quad
  C' = 03'
\end{equation*}
and a corresponding basis for \( \rcomp{L^2\lambda^{0,1}} \) is
\begin{alignat*}{3}
  P &= 23-2'3',\quad& Q &= 31-3'1',\quad& R &= 12-1'2', \\
  P' &= 23'-32',& Q' &= 31'-13',& R' &= 12'-21'.
\end{alignat*}
We put \( E = 11'+22'+33' \), and note that this is a generator of the
central \( \un(1) \) in \( \un(3) \).  Then \( \base{E,A,\dots,R'} \)
is a basis for \( \lie n = T_{\Id\SU(3)}N \) and \( \base{A,\dots,R'}
\) is a basis for~\( \lie m = T_{\Id\Un(3)}M \).  We use lower case
letters to denote the corresponding dual bases of~\( \lie n^* \) and
\( \lie m^* \).  These give left-invariant one-forms on \( \SO(7) \),
with \( da(X,Y) = - a([X,Y]) \) for \( X,Y \in \so(7) \), etc.  We write
\begin{equation*}
  d_Na = (da)|_{\Lambda^2\lie n} \quad\text{and}\quad
  d_Ma = (da)|_{\Lambda^2\lie m} 
\end{equation*}
at \( \Id \in \SO(7) \).  For a left-invariant form \( \alpha \in
\Omega^k(\SO(7)) \), we have at~\( \Id \in \SO(7) \) that \( d\alpha =
d_N\alpha \) if \( \alpha \) is right \( \SU(3) \)-invariant and \(
d\alpha = d_M\alpha \) if \( \alpha \) is right \( \Un(3)
\)-invariant.  For our choice of bases, we have
\begin{align*}
  d_Ma &= - b\wedge r + c\wedge q - b' \wedge r' + c'\wedge q',&
  d_Mp &= - \tfrac12(b\wedge c - b' \wedge c'),\\
  d_Ma' &= - b\wedge r' + c\wedge q' + b' \wedge r - c'\wedge q,&
  d_Mp' &= - \tfrac12(b\wedge c' + b' \wedge c)
\end{align*}
with the other derivatives  obtained by applying the cyclic
permutation \( (a,a',p,p') \to (b,b',q,q') \to (c,c',r,r') \to
(a,a',p,p') \).  We use \( \cyc \) to denote sums over
this group of permutations.

The two-form \( \omega_I \) of~\eqref{eq:forms} is
\begin{equation*}
  \begin{split}
    \omega_I &= \lambda(a'\wedge a + b'\wedge b + c'\wedge c) +
    \mu(p'\wedge p + q'\wedge q + r'\wedge r) \\
    &= \cyc(\lambda a' \wedge a + \mu p' \wedge p).
  \end{split}
\end{equation*}
On \( N \), we have the forms \( \hat\omega_J \) and \(
\hat\omega_K \) given by
\begin{equation*}
  \hat\omega_J + i \hat\omega_K = \cyc \bigl((p+ip')\wedge (a+ia')\bigr).
\end{equation*}
Choosing a local section \( s \) of \( \pi\colon N \to M \) 
such that \( s(\Id\Un(3)) = \Id\SU(3) \) and \( s^*e = 0 \), we then obtain
local two-forms
\begin{equation*}
 \omega_J = s^*\hat\omega_J,\quad \omega_K = s^*\hat\omega_K
\end{equation*}
completing the triple of~\eqref{eq:forms}.  The corresponding metric
on~\( M \) is
\begin{equation}
  \label{eq:g}
  g = \cyc(\lambda(a^2 + {a'}^2) + \mu(p^2 + {p'}^2))
\end{equation}
and condition~\eqref{eq:normalisation} is simply
\begin{equation}
  \label{eq:norm}
  \lambda\mu = 1.
\end{equation}
These are the only invariant metrics on~\( M \) with normalised volume
form, since \( TM \)~\eqref{eq:TC} has exactly two irreducible
summands.

At \( \Id\Un(3) \), the almost complex structures satisfy
\begin{gather*}
  IA = A',\quad IP = P',\quad
  J\tfrac1{\sqrt\lambda}A = \tfrac1{\sqrt\mu}P,\quad
  J\tfrac1{\sqrt\lambda}A' = -\tfrac1{\sqrt\mu}P',\\
    K\tfrac1{\sqrt\lambda}A = \tfrac1{\sqrt\mu}P',\quad
  K\tfrac1{\sqrt\lambda}A' = \tfrac1{\sqrt\mu}P.
\end{gather*}
These act on forms via \( Ia = -a(I\cdot) \), so with the
normalisation condition~\eqref{eq:norm}, we have \( Ja = \mu p \), \(
Jp = -\lambda a \), etc.

\begin{lemma}
  \label{lem:all-aqH}
  These describe all invariant almost quaternion\hyphen Hermitian
  structures on~\( M \) with normalised volume form.
\end{lemma}

\begin{proof}
  We have noted above that \eqref{eq:g} gives all the invariant
  metrics.  Now the local almost complex structures, or equivalently
  their Hermitian two forms, associated to the almost quaternion
  Hermitian structure span a \( \Un(3) \)-invariant subspace~\( V \)
  of \( \Lambda^2T \) of dimension~\( 3 \).  Counting dimensions in
  the decomposition~\eqref{eq:L2T}, shows that \( V \) is a subspace
  of \( \bR\omega_0 + \bR\tilde\omega_0 + \rcomp{L^3} \).  In
  particular, \( V \cap \rcomp{L^3} \) is at least one-dimensional; \(
  \Un(3) \)-invariance implies that \( \rcomp{L^3} \leqslant V \).  As
  \( \omega_J \) and \( \omega_K \) are \( g \)-orthogonal of the same
  length for each normalised \( g \) in~\eqref{eq:g}, we see that \( J
  \) and \( K \) are local almost complex structures belonging to the
  almost quaternion\hyphen Hermitian geometry.  Finally, \( I = JK \)
  is specified too.
\end{proof}

\begin{lemma}
  \label{lem:d-om}
  For the choices of \( \omega_I \), \( \omega_J \) and \( \omega_K \)
  above normalised by~\eqref{eq:norm} we have at the base point \(
  \Id\Un(3) \in M \) that
  \begin{gather*}
    Id\omega_I = Id_M\omega_I =  (\tfrac12\mu - 2\lambda)\Phi,\\
    Jd\omega_J = 2\lambda\Phi - \tfrac12\mu^3\Psi,\quad Kd\omega_K =
    2\lambda\Phi + \tfrac12\mu^3\Psi,
  \end{gather*}
  where
  \begin{gather*}
    \Phi = \cyc (a\wedge b\wedge r - a' \wedge b' \wedge r + a \wedge
    b' \wedge r' + a' \wedge b \wedge r'),\\
    \Psi = \cyc(p \wedge q \wedge r - 3 p \wedge q' \wedge r')
  \end{gather*}
  and \( A d\omega_A(\cdot,\cdot,\cdot) = -
  d\omega_A(A\cdot,A\cdot,A\cdot) \), for \( A = I, J, K \).
\end{lemma}

\begin{proof}
  As \( \omega_I \) is \( \Un(3) \)-invariant we have \( Id\omega_I =
  Id_M\omega_I \) which equals
  \begin{equation*}
    (2\lambda-\tfrac12\mu)I\cyc (a\wedge b' \wedge r + a'\wedge b
    \wedge r - a \wedge b \wedge r' + a' \wedge b' \wedge r')
  \end{equation*}
  and gives the first claimed formula valid at any point of~\( M \).

  For our choice of section~\( s \), we have at \( \Id\Un(3) \) that
  \( Jd\omega_J = J s^*d_N\tilde\omega_J = Jd_M\tilde\omega_J \) which
  is
  \begin{equation*}
    J\cyc\Bigl(-\tfrac12 a\wedge b\wedge c + \tfrac32 a \wedge
    b'\wedge c' +2(a\wedge q\wedge r - a\wedge q'\wedge
    r' + a'\wedge q\wedge r' +
    a'\wedge q'\wedge r)\Bigr).
  \end{equation*}
  Combined with the description of~\( J \), we thus get the claimed
  formula.  The computation for \( Kd\omega_K \) is similar.
\end{proof}

To compute the intrinsic torsion we use the \enquote{minimal
description} of \cite{Cabrera-S:aqH-torsion} which relies on computing
the forms \( \beta_I = Jd\omega_J + Kd\omega_K \), etc., and the
contractions \( \Lambda_A\beta_B \) of \( \beta_B \) with \( \omega_A
\).  For our structures, we have at the base point
\begin{equation*}
  \beta_I = 4\lambda\Phi,\quad
  \beta_J = \tfrac12(\mu\Phi + \mu^3\Psi),\quad \beta_K =
  \tfrac12(\mu\Phi - \mu^3\Psi)
\end{equation*}
and all contractions \( \Lambda_A\beta_B = 0 \).  This confirms that
the intrinsic torsion~\( \xi \) has no components in \( \real{E(S^3H+H)} \).  

\begin{theorem}
  \label{thm:xi}
  The component of \( \xi \) in \( \real{\Lambda^3_0ES^3H} \) is always
  non-zero, so the almost quaternion\hyphen Hermitian is never
  quaternionic.  The component of \( \xi \) in \( \real{\Lambda^3_0EH} \) is
  zero if and only if \( 2\lambda = \mu \).
\end{theorem}

\begin{proof}
  Since we have shown in \S\ref{sec:intrinsic-torsion} that \( \xi \)
  has no component in \( \real{K(S^3H + H)} \) and we saw above that
  each one form \( \Lambda_A\beta_B \) is zero, at the base point, the
  results of \cite{Cabrera-S:aqH-torsion} show that the \(
  \Lambda^3_0ES^3H \)\hyphen component of \( \xi \) corresponds to
  \begin{equation*}
    \psi^{(3)} \coloneqq \tfrac1{12}(\beta_I+\beta_J+\beta_K) =
    \tfrac1{12}(4\lambda+\mu)\Phi 
  \end{equation*}
  which is always non-zero under condition~\eqref{eq:norm}.  The
  component in \( \Lambda^3_0EH \) is determined by
  \begin{equation*}
    \psi^{(3)}_I \coloneqq \tfrac1{8}(-\beta_I + 2(3 + \mathcal
    L_I)\psi^{(3)}), 
  \end{equation*}
  where \( \mathcal L_I = I_{(12)} + I_{(13)} + I_{(23)} \), with \(
  I_{(12)}\alpha = \alpha(I\cdot,I\cdot,\cdot) \), etc.  Now \(
  \mathcal L_I\Phi = \Phi \), so
  \begin{equation*}
    \psi^{(3)}_I = \tfrac1{12}(\mu-2\lambda)\Phi
  \end{equation*}
  and the result follows.
\end{proof}

\begin{corollary}
  \label{cor:quaternionic}
  The invariant almost quaternion-Hermitian structures on~\( M \) are
  not quaternionic integrable, and their quaternionic twistor spaces
  are not complex.
\end{corollary}

\begin{proof}
  This follows directly from the following two facts
  \cite{Salamon:quaternionic}:
  \begin{inparaenum}
  \item The underlying quaternionic structure is integrable if and
    only if the intrinsic torsion \( \xi \) has no \( S^3H \)
    component, i.e. it lies in \( (\Lambda^3_0E+K+E)H \).
  \item The quaternionic twistor space is complex if and only if the
    underlying quaternionic structure is integrable.
  \end{inparaenum}
  But we have shown the \( \Lambda^3_0ES^3H \)-component of \( \xi \)
  is non-zero, so the result follows.
\end{proof}

The almost Hermitian structure \( (g,\omega_I) \) is easily seen to be
integrable: \( d_M(a+ia') = -(b-ib') \wedge (r+ir') +
(c-ic')\wedge(q+iq') \in \Lambda^{1,1}_I \), \( d_M(p+ip') =
-\tfrac12(b+ib')\wedge(c+ic') \in \Lambda^{2,0}_I \).  In addition,
from Lemma~\ref{lem:d-om}, we see that \( d\omega_I \) is orthogonal to \(
\omega_I\wedge\Lambda^1 \).  It follows that \( d\omega_I \) is
primitive.  

Now recall that Gray and Hervella \cite{Gray-H:16}, showed that the
intrinsic torsion of an almost Hermitian structure \( (g,\omega) \)
lies in
\begin{equation*}
  \mathcal W = \mathcal W_1 + \mathcal W_2 + \mathcal W_3 + \mathcal
  W_4 = \rcomp{\Lambda^{3,0}} + \rcomp{U^{3,0}} +
  \rcomp{\Lambda^{2,1}_0} + \rcomp{\Lambda^{1,0}},
\end{equation*}
with \( U^{3,0} \) irreducible: the \( \mathcal W_1 + \mathcal W_2
\)-part is determined by the Nijenhuis tensor; the \( \mathcal W_1 +
\mathcal W_3 + \mathcal W_4 \)-part by \( d\omega \).  We now have
from Lemma~\ref{lem:d-om}:

\begin{proposition}
  The Hermitian structure \( (g,\omega_I,I) \) is of Gray-Hervella
  type \( \mathcal W_3 \), except when \( 4\lambda=\mu \),
  when it is Kähler.  Furthermore, the Kähler metric is symmetric.
\end{proposition}

Note that the Kähler parameters do not correspond to the parameters in
Theorem~\ref{thm:xi} that give \( \xi \in \real{\Lambda^3_0ES^3H} \).

\begin{proof}
  It remains to prove the last assertion.  As in \cite{Salamon:tour},
  note that \( \SO(7)/\Un(3) \allowbreak \cong \SO(8)/\Un(6) \cong
  \SO(8)/\SO(2)\SO(6) \), which is the quadric.  The latter is
  isotropy irreducible and carries a unique \( \SO(8) \)-invariant
  metric with fixed volume, which is Hermitian symmetric so Kähler.
  However, we have seen that there is a unique Kähler metric with the
  same volume invariant under the smaller group~\( \SO(7) \), so these
  Kähler metrics must agree.
\end{proof}

\begin{remark}
  Each \( \SO(7) \)-invariant metric~\( g \) on~\( M \) is given
  by~\eqref{eq:g} and so is a Riemannian submersion over~\( \CP(3) \)
  with fibre~\( S^6 \).  The standard theory of the canonical
  variation~\cite{Besse:Einstein} tell us that precisely two of
  these metrics are Einstein.  One is the symmetric case \( 4\lambda =
  \mu \).  The other is when \( 8\lambda = 3\mu \), as verified by
  Musso~\cite{Musso:twistor} in slightly different notation.  Again
  these particular parameters are not those for which \( \xi \) is
  special.
\end{remark}

\begin{remark}
  It can be shown that the local almost Hermitian structures \(
  (g,\omega_J,J) \) and \( (g,\omega_K,K) \) above are each of strict
  Gray-Hervella type \( \mathcal W_1 + \mathcal W_3 \) at the base
  point, unless \( 4\lambda = 3\mu \), when they have type \( \mathcal
  W_1 \).  In particular, the Nijenhuis tensors \( N_J \) and \( N_K
  \) are skew-symmetric at the base point and equal to \(
  \tfrac16(4\lambda+\mu)(3\Phi \mp \mu^2\Psi) \) at~\( \Id\Un(3) \).
  In \cite{Cabrera-S:aqH-torsion} we showed how \( N_I \) is
  determined by \( Jd\omega_J - Kd\omega_K \).  In this case, we have
  the interesting situation that this latter tensor is non-zero, even
  though \( N_I \) vanishes.  Using \cite{Alekseevsky-M:subordinated},
  one can prove that the obstruction to quaternionic integrability is
  proportional to \( N_I+N_J+N_K = (4\lambda+\mu)\Phi \), confirming
  that this is non-zero and the results of
  Corollary~\ref{cor:quaternionic}.
\end{remark}

\providecommand{\bysame}{\leavevmode\hbox to3em{\hrulefill}\thinspace}
\providecommand{\MR}{\relax\ifhmode\unskip\space\fi MR }
\providecommand{\MRhref}[2]{%
  \href{http://www.ams.org/mathscinet-getitem?mr=#1}{#2}
}
\providecommand{\href}[2]{#2}

\end{document}